\documentclass[10pt]{article}
\usepackage[centertags]{amsmath}
\usepackage{amsfonts}
\usepackage{amssymb}
\usepackage{amsthm}
\newcommand{\ex}{{\bf\sf E}}               

\newcommand{\call}{{\cal L}}

\newcommand{\calp}{{\cal P}}

\newcommand{\calk}{{\cal K}}
\newcommand{\cals}{{\cal S}}

\newcommand{\calf}{{\cal F}}

\newcommand{\al}{\alpha}                
\newcommand{\s}{\sigma}               

\newcommand{\ra}{\rightarrow}           

\newtheorem{thm}{Theorem}
\newtheorem{lem}{Lemma}

\newcommand{\Real}{\mathbb R}

\begin{document}

\baselineskip20pt

\title{Local and Global Existence of A Nonlocal Equation with A Singular Integral Drift Term}
\author{Yingdong Lu \\Mathematical Sciences\\ IBM T.J. Watson Research Center \\Yorktown Heights, NY 10598}
\date{}
\maketitle

\begin{abstract}
We study an initial value problem with fractional Laplacian and a singular drift term, and obtain local and global existence theorems similar to the results in \cite{Jourdain2005APA}.

\end{abstract}

\section{Introduction}
\label{sec:intro}

Consider the following initial value problem,  
\begin{align}
\label{eqn:main}
\left\{ \begin{array}{ccc}\partial_t u(t, x) & = &-(-\Delta)^{\al/2} u -\nabla \cdot (uB(u)),\\
u(0,x) &=&u_0(x), \end{array}  \right.
\end{align}
where $u:{\mathbb R}_+ \times {\mathbb R}^d\rightarrow {\mathbb R}$ for positive integer $d$ and $\al\in(1,2)$. 
The operator  $ -(-\Delta)^{\al/2}$ is the fractional power of the Laplacian $\Delta$, analytically, it can be defined as, 
\begin{align}
\label{eqn:fracLapAna}  -(-\Delta)^{\al/2} v(x) = \calf^{-1} (|\xi|^\al \calf(v) (\xi)) (x),
\end{align}
for any Schwartz function $v\in \cals$, with $\calf$ denoting the Fourier operator. Probabilistically, it can be also viewed as a Markov jump process operator, thus, has the following equivalent form, 
\begin{align}
\label{eqn:fracLapProb}  -(-\Delta)^{\al/2} v(x) =K\int_{{\mathbb R}^d} (v(x+y) -v(x) -\nabla v(x) \cdot  y{\bf 1}_{|y|\le 1}) \frac{dy} {|y|^{d+\al}},
\end{align}
where $K=K_{\al,d}$ is a constant. $B(u)$ is a singular integral operator defined by,
\begin{align}
\label{eqn:integralOp}
B(u) (x) = \int_{{\mathbb R}^d} b(x,y) u(y) dy,
\end{align}
with a Calder\'on-Zygmund singular integral kernel $b(x,y)$.  Recall the conditions that a Calder\'on-Zygmund kernel has to satisfy:  
there are constants $C$ and $\delta>0$, such that for any $x,y \in \Real^d$, 
\begin{align*}
	|b(x,y)| &\le \frac{C}{|x-y|^d}, \\
	|b(x,y)-b(x', y)| &\le \frac{C|x-x'|}{(|x-y|+|x'-y|)^\delta}, \quad \text{whenever $|x-x'| \le \frac12 \max (|x-y|, |x'-y|)$} \\
	|b(x,y')-b(x, y)| &\le \frac{C|y-y'|}{(|x-y|+|x'-y|)^\delta}, \quad \text{whenever $|y-y'| \le \frac12 \max (|x-y|, |x'-y|)$}
\end{align*}
For more detailed analysis on integration with respect to Calder\'on-Zygmund kernel, see, e.g.,  ~\cite{stein1993harmonic}. 
We also denote the singular integral operator as $\calk_b$, hence, write $B(u) (x)=\calk_b u(x)$.

Equation \eqref{eqn:main}, in various forms, has been studied in both mathematical and physics literature. In~\cite{MANNJR2001159}, 
one form of this equation characterizes fractal interfaces in statistical mechanics in the presence of self-similar hopping surface diffusion, 
and generalizes the classical Kardar–Parisi–Zhang (KPZ) model. Regularity and conservation laws for \eqref{eqn:main} with different drift $B(u)$, are considered 
in~\cite{doi:10.1137/S0036139996313447, BILER2001613, Jourdain2005APA}. When $B(u)$ is a more regular operator, where $b(x,y)$ is a convolutional kernel satisfy necessary bound on value and 
derivative such that the integral operator is $(p,\infty)$ ($(p,q)$ refers to bounded operator from $L_p$ to $L_q$) for 
some $p> d/\beta$, the local and global existences of \eqref{eqn:main} are obtained in \cite{Jourdain2005APA}. In this paper, we deal with the case where $B(u)$ is represented by a general Calder\'on-Zygmund operator. The integration is singular, the boundedness of the operator is weaker. Our results consists of identifying the function spaces in which local and global existences results of \eqref{eqn:main} can be derived. 

In Sec. \ref{sec: materials}, we will provide necessary background material and notations; in Sec. \ref{sec:local}, we will present the local existence results; and the global existence results will be presented in Sec. \ref{sec:global}.

\section{Notation and Basic Formulas}
\label{sec: materials}

\subsection{Function spaces and norms}

For $0<p<\infty$, the Calder\'on-Zygmund operator $\calk_b$ is known to be $(p,p)$, i.e. bounded operator maps $L^p$ functions to $L^p$ functions. More precisely, there exists a constant $A_p$, such that $\|\calk_b f\|_p \le A_p\|f\|_p$ for any function $f\in L^p(\Real^d)$ with the $L^p$ norm $\|\cdot \|_p$ defined as $ \| f\|_p:=( \int_{\Real^d} |f(x)|^p dx)^{1/r}$, and the space $L^p(\Real^d)$ includes all the measurable functions that has finite $L^p$ norm. In addition, the operator $\calk_b$ is bounded on Lipschitz space $Lip(\epsilon)$ for any $\epsilon\le \delta$ with $\delta$ being the parameter in the Calder\'on-Zygmund kernel $K1=0$ by the main theorem (Theorem 1.6) in~\cite{ZhengLiTao2019}. Recall that $Lip(\al)$ refers the space of functions satisfies $|f(x)-f(y)|\le Cd(x,y)^\al$.

The function space in which we will derive the local existence theorem is defined to be $L_x^pL_t^\infty(\Real^d\times [0,t])$ the space of all functions whose $L_x^pL_t^\infty$ norm is finite, with 
\begin{align*}
||f(x,t) \|_{L_x^pL_t^\infty} := \sup_{0\le s\le t} \left(\int_{x\in\Real^d} |f(x,s)|^p dx\right)^{1/p}.
\end{align*}

\subsection{Weak Solution}

The weak solution, for any test function $\psi(x, s)$ in $C^\infty_0({\mathbb R}^d\times {\mathbb R}_+ )$, a function $u_t(x)$ in a suitable function space is a weak solution if the following is satisfied, 
\begin{align}
\int_{{\mathbb R}^d}  \psi(x,t)u_t(x) dx - \int_{{\mathbb R}^d}  \psi(x, 0)u_0(x) dx =& \int_0^t 
\int_{{\mathbb R}^d}  \left[\frac{\partial}{\partial s}\psi(x, s)- (-\Delta)^{\al/2}\psi(s,x) + B(u_s)(x). \nabla \psi(x,s)\right] u_s(x) dx ds. \label{eqn:weakSoln}
\end{align}

\subsection{Semigroup and generator}
The generator for semi-group $\exp(-t(-\Delta)^{\al/2}) $ is denoted as $p^\al_t$. For any smooth function $\phi$, then $\Psi(s, x) = p_{t-s}^\al \star \phi(x)$ satisfies,
\begin{align*}
	\frac{\partial }{\partial s} \Psi(s,x) - (-\Delta)^{\al/2} \Psi(\s,x)=0.
\end{align*} 
It is known that, see, e.g. \cite{Jourdain2005APA},
\begin{lem} 
\label{lem:operator_p}
	If $m \ge q\ge 1$, and $f\in L^q$, then,
	\begin{align*}
		\| p^\al_t \star f  \|_m &\le C t^{-\frac{d}{\al}\left(\frac{1}{q}-\frac{1}{m}\right)} \|f \|_q;\\
		\| \nabla p^\al_t \star f  \|_m &\le C t^{-\frac{d}{\al}\left(\frac{1}{q}-\frac{1}{m}\right)-\frac{1}{\al}} \|f \|_q;
	\end{align*}
\end{lem}

\section{Local Existence}
\label{sec:local}

The following version of the Banach fixed point theorem, can be found in e.g.~\cite{cannone1995ondelettes},
\begin{lem}
	\label{lem:fixed_point}
Suppose that $B:X\times X\rightarrow X$ is a bilinear mapping for a Banach space $(X, \|\cdot\|_X)$, and it satisfies, 
\begin{align*}
||B(x_1, x_2)||_X \le \eta || x_1||_X ||x_2||_X,
\end{align*}
for some $\eta>0$. Then, for each $y \in X$ satisfying $4\eta||y||<1$, equation,
\begin{align*}
x=y+B(x,x)
\end{align*}
admits a unique solution $x$ in the ball $\{ z\in X, ||z||\le R\}$  with $R=\frac{1-\sqrt{1-4\eta ||y||}}{2\eta}$. Moreover, the solution satisfies inequality $||x||_X \le 2||y||_X$.
\end{lem}

Previous results assume that the kernel 
$|b(x,y)| \le C|x|^{\beta-d}$
for some $\beta>0$. This will lead to $||B(u)||_\infty \lesssim  ||u||$. 
In order to apply the fixed point theorem Lemma \ref{lem:fixed_point}, define the following bilinear map,
\begin{align}
\label{def:bilinear}
	B(u,v) (t,x) = \int_0^t \nabla p_{t-s}^\al \star (B(v_s) u_s) ds.
\end{align}
\begin{lem}
\label{lem:boundness}
For any $p\ge 2$, there exists a constant $C>0$, such that, 
\begin{align*}
	& \left\| \int_0^t \nabla p_{t-s}^\al \star (B(v_s) u_s) ds\right\|_{L_x^pL_T^\infty} \le C  T^{1-\frac{1}{\al}}  \left\|v\right\|_{L_x^{p}L_T^\infty} \left\|u \right\|_{L_x^{p}L_T^\infty}.  
\end{align*}
\end{lem}
\begin{proof}
\begin{align*}
	 \left\| \int_0^t \nabla p_{t-s}^\al \star (B(v_s) u_s) ds\right\|_{L_x^pL_T^\infty}  \stackrel{(a)}{\le} & 
	  \int_0^T \left\|\nabla p_{t-s}^\al \star (B(v_s) u_s) \right\|_{L_x^p}  ds  \\ \stackrel{(b)}{\le} & C \int_0^T (t-s)^{-\frac{d}{\al} \left(\frac{1}{p}\right)-\frac{1}{\al}} \left\|(B(v_s) u_s) \right\|_{L_x^{\frac{p}{2}}}  ds   \\ \stackrel{(c)}{\le}  & C \int_0^T(t-s)^{-\frac{d}{\al} \left(\frac{1}{p}\right)-\frac{1}{\al}} \left\|(B(v_s) \right\|_{L_x^p} \left\|u_s \right\|_{L_x^{p}}   ds  
	  \\ \stackrel{(d)}{\le} & C' \int_0^T (t-s)^{-\frac{d}{\al} \left(\frac{1}{p}\right)-\frac{1}{\al}} \left\|v_s \right\|_{L_x^{p}} \left\|u_s \right\|_{L_x^{p}}   ds   \\ \stackrel{(e)}{\le} & C'  T^{1-\frac{1}{\al}}  \left\|v\right\|_{L_x^{p}L_T^\infty} \left\|u \right\|_{L_x^{p}L_T^\infty}	  	  
\end{align*}
where (a) is due to the nonnegativity of the integrand; (b) comes from the properties of operator $p_t^\al$ in Lemma \ref{lem:operator_p} with $m=p, q=\frac{p}{2}$; (c) is an application of the H\"older's inequality; (d) follows from the boundedness of the Calder\'on-Zygmund operator; and (e) is the result of a simple integration calculation.
\end{proof}
Then, apply the fixed point theorem, Lemma \ref{lem:fixed_point}, the following local existence theorem can be established. 
\begin{thm}
For any $p>2$, and $u_0\in L_x^p$, there exist a a constant $T^*>0$ and function $u(x,t) \in L_x^pL_{T^*}^\infty$ such that $u(x,t)$ is a weak solution, in the sense of \eqref{eqn:weakSoln}, to \eqref{eqn:main}.
\end{thm}
\begin{proof}
Lemma \ref{lem:boundness} indicates that the bilinear form defined in \eqref{def:bilinear} satisfies the condition for the fixed point theorem, Lemma \ref{lem:fixed_point}. Hence, we can conclude that, for any function $u_0\in L_x^pL_{T^*}^\infty$, there exist a constant $T^*>0$ and a weak solution $u(x,t) \in L_x^pL_{T^*}^\infty$ in the sense that is defined in \eqref{eqn:weakSoln}.
\end{proof}

\section{Global Existence}
\label{sec:global}

The goal of this section is to establish the global existence of \eqref{eqn:main} via a combined probabilistic and analytic argument. Precisely, global existence means that for any given time horizon $T< \infty$, the solution $u(t,x)$ solves the equation in a weak sense.  The basic approach is a probabilistic one, we will construct a solution to a stochastic differential equation, \eqref{eqn:main_sde}, that will be defined below. It has been demonstrated, see e.g.  \cite{Jourdain2005APA} that the density function of this solution solves \eqref{eqn:main}. The solution of  \eqref{eqn:main_sde} will be constructed through an iterative procedure, with key steps obtained utilizing the properties of the Calder\'on-Zygmund operator. 

We start with the following lemma on the existence and uniqueness of a class of stochastic differential equations defined by a stable process,
\begin{lem}
	Given initial condition $X_0$,  $\al$-stable process $S_t$,  and a bounded function $a_t:{\mathbb R} \rightarrow {\mathbb R}^d$, and $a_t$ is Lipschitz for any $t\in [0, T]$, the following stochastic differential equation, 
	\begin{align}
		X_t = X_0  + S_t + \int_0^t a_s(X_s) ds,
	\end{align}
has a unique (pathwise and in law) solution in the  Skorohod space $D([0,T], {\mathbb R}^d)$.
\end{lem}
\begin{proof}
When $a_s$ is constant over time, this lemma is stated and proved in~\cite{Jourdain2005APA}. Similarly, here, the pathwise existence and uniqueness are the result of a fixed point argument. Then the Yamada-Watanabe Theorem, see e.g. ~\cite{ikeda2014stochastic} ensures the uniqueness in probability law. 
\end{proof}

Note that we are studying the probability measure of the paths, so let us first introduce the Skorohod space $D([0, T], {\mathbb R}^d)$, the set of c\'adl\'ad functions (functions that are right-continuous and have left limits everywhere) from $[0,T]$ to ${\mathbb R}^d$. Let  $\calp_T$ denote the set of all probability measures on $D([0, T], {\mathbb R}^d)$ that are absolutely continuous with respect to the Lebesgue measure, and
\begin{align*}
	{\tilde \calp}_T =\left\{ P\in \calp_T, P_0=\frac{|u_0|}{||u_0||_1}\right\},
\end{align*}
i.e. the subset of $\calp_T$ with the initial condition fixed. Define the following metric on ${\tilde \calp}_T$, for any $p>0$,
\begin{align*}
d_{T,p}(m_1, m_2) = \max\{ \rho_{p, T}(m_1, m_2), || f_1-f_2||_p \}
\end{align*}
where $ \rho_{p, T}$ denotes the $p$-Wasserstein distance on ${\tilde \calp}_T$, i.e., 
\begin{align*}
	\rho_{p, T}(P,Q) =\left\{\inf\int_{D_T\times D_T }\left[\sup_{t\le T} |x(t)-y(t)|\wedge 1\right]^pR(dx, dy)\right\}^{\frac{1}{p}},
\end{align*}
with the infimum is taken over measures $R$ with marginals $P$ and $Q$, and $R(x(0)=y(0) ) =1$. And $f_i$ denotes the density function of $m_i$ for $i=1,2$. 
It is worth noting that the metric space $(\calp_T, d_{T,p})$ is complete.  The stochastic stochastic differential equation is the following one,
\begin{align}
\label{eqn:main_sde}
X_t = X_0 + S_t + \int_0^t \int_{R^d} b(X_s, y) {\tilde P}_t(dy) ds.
\end{align}
with $ {\tilde P}_t(dy)$ denotes the probability law of $X_t$. 

\subsection{An operator defined on the space of probability measures}

Define an operator $\Psi^Y$ on ${\tilde \calp}_T$, indexed by an process $Y(t)$ as follow.  
For each $m\in {\tilde \calp}_T$, denote $ m^Y_t$ as the $Y_t$-weighted version of $m$, more specifically,   
\begin{align*}
 m^Y_t(A) = \|u_0\|_1\ex[{\bf 1}_A(Y_t) \text{sgn} (u_0(Y_0) ) ].
 \end{align*}
Thus, $\Psi_t^Y(m)$ is the probability measure induced by the solution to the following stochastic differential equation,
\begin{align*}
X_t = X_0 + S_t + \int_0^t \int_{R^d} b(X_s, y) m^Y_t(dy) ds.
\end{align*}
The following result is the key estimation for the global existence.  
\begin{lem}
	\label{lem:concentration_one}
	With the condition that $\delta>1$ and $\calk_b 1=0$, we can conclude that, 
There exists a constant $C>0$, such that, for any $m_1, m_2\in {\hat \calp}_T$, $d_{t,p}(\Psi_t^Y(m^1), \Psi_t^Y(m^2))\le C  \int_0^t d_{s,p}(m^1,m^2)ds$.  
\end{lem}
\begin{proof}
First of all, we know that, for $p>1$,
\begin{align*}
	\rho_{p, t} ( \Psi_t^Y(m^1), \Psi_t^Y(m^2)) & \le \left(\ex\left[\sup_{s\le t}|X_s^1-X_s^2|^p\right]\right)^{\frac{1}{p}}\\ &\le \left(\ex\left[\int_0^t\Big|\int_{R^d} b(X^1_s, y)m^1(dy)-b(X^2_s, y)m^2(dy) \Big|^p\right]\right)^{\frac{1}{p}}
\end{align*}
Similarly,  from relation between density and the process, see, e.g., \cite{olivera2019}, we can conclude that, 
\begin{align*}
	\|f_1-f_2\|_p& \le \left(\ex\left[\sup_{s\le t}|X_s^1-X_s^2|^p\right]\right)^{\frac{1}{p}}\\ &\le \left(\ex\left[\int_0^t\Big|\int_{R^d} b(X^1_s, y)m^1(dy)-b(X^2_s, y)m^2(dy) \Big|^p\right]\right)^{\frac{1}{p}}
\end{align*}
Hence, we have, 
\begin{align*}
	d_{p, t} ( \Psi_t^Y(m^1),\Psi_t^Y(m^2)) & \le \left(\ex\left[\sup_{s\le t}|X_s^1-X_s^2|^p\right]\right)^{\frac{1}{p}}\\ &\le \left(\ex\left[\int_0^t\Big|\int_{R^d} b(X^1_s, y)m^1(dy)-b(X^2_s, y)m^2(dy) \Big|^p\right]\right)^{\frac{1}{p}}
\end{align*}
By triangular inequality, it suffices to establish the following two inequalities, 
\begin{align}
\label{eqn:first_inq}
	\left(\ex_\Pi\Big| \int_0^t \int_{\Real^d} b(X^1_s, y) m^1_t(dy) ds-\int_0^t \int_{\Real^d} b(X^2_s, y) m^1_t(dy)  ds\Big| ^p \right)^{1/p} \le C  \int_0^t d_{s,p}(m^1,m^2)ds,
\end{align}
and
\begin{align}
\label{eqn:second_inq}
	\left(\ex_\Pi\Big| \int_0^t \int_{\Real^d} b(X^2_s, y) m^1_t(dy) ds-\int_0^t \int_{\Real^d} b(X^2_s, y) m^2_t(dy)  ds\Big| ^p \right)^{1/p}\le C  \int_0^t d_{s,p}(m^1,m^2)ds.
\end{align}
The inequality \eqref{eqn:first_inq} follows from the Lipschitz property of the Calder\'on-Zygmund operator, again (Theorem 1.6) in~\cite{ZhengLiTao2019}. The inequality \eqref{eqn:second_inq} holds because of  the boundedness of the Calder\'on-Zygmund operator. More specifically, 
\begin{align*}
	&\left(\ex_\Pi\Big| \int_0^t \int_{\Real^d} b(X^2_s, y) m^1_t(dy) ds-\int_0^t \int_{\Real^d} b(X^2_s, y) m^2_t(dy)  ds\Big| ^p \right)^{1/p} \\ 
	\le &\left(\ex_\Pi\int_0^t  \Big|\int_{\Real^d} b(X^2_s, y) f^1- f_2 dy  \Big| ^p ds \right)^{1/p} \\ \le &\int_0^t  \left(\ex_\Pi\Big|\int_{\Real^d} b(X^2_s, y) f^1- f_2 dy  \Big| ^p\right)^{1/p} ds \\
	\le&   C  \int_0^t d_{s,p}(m^1,m^2)ds
\end{align*}
\end{proof}

\subsection{Global Existence}

\begin{thm}
	Under the following conditions, we will have the global existence of the solution to equation \eqref{eqn:main}.
	\begin{itemize}
		\item[i.]
		The Calder\'on-Zygmund operator with $\delta>1$, and $\calk_b1=0$.
		\item[ii.]
		the initial condition $u_0$ is a Lipschitz function, in other word, $u_0\in Lip(1)$. 
	\end{itemize}
\end{thm}
\begin{proof}
Start with an arbitrary stochastic process $Y^0_t \in D([0,T], \Real^d)$, for any $n\ge 1$, define, $Y^n_t:= \Psi^{Y^{n-1}_t}$. Then, iterating the inequality in Lemma \ref{lem:concentration_one}, we get,
\begin{align*}
	d_{T,p}( \Psi_T^n(m^1), \Psi_T^n(m^2))\le  \frac{C^n}{n!} d_{T,p} (\Psi_T^0(m^1),\Psi_T^0(m^2)).
\end{align*}
This means that $Y_n^t$ is a Cauchy sequence in $d_{T,p}$, hence there exists a limit, denoted as $Y_\infty^t$. It can be verified that $Y_\infty^t$ is a strong solution to the stochastic differential equation \eqref{eqn:main_sde}. Proposition 4.4 in \cite{Jourdain2005APA} confirms that this solution will produce a density function that is a solution to equation \eqref{eqn:main}. 
\end{proof}
\bibliographystyle{abbrv}
\bibliography{Lu}

\end{document}